\documentclass[11pt]{article}
\usepackage{amssymb}
\usepackage{mathrsfs,pstricks}
\usepackage{amsfonts}
\usepackage{latexsym}
\usepackage{amsmath,amssymb}
\usepackage{amsthm}
\usepackage{amsfonts}
\usepackage{ifthen,calc}
\usepackage{color}
\usepackage{enumerate}
\usepackage{tikz}
\usepackage{epsfig}
\usepackage{graphicx}
\usepackage{float}
\usepackage{latexsym}
\usepackage{multicol}
\usepackage{subfigure}
\usepackage{geometry}
\usepackage{setspace}
\usepackage{CJK}
\usepackage{mathrsfs}
\usepackage{amsmath}
\usepackage{amssymb,mathrsfs,amsthm,amsfonts}
\usepackage{graphicx}
\usepackage[pdfstartview=FitH]{hyperref}
\usepackage[title]{appendix}
\usepackage{color}
\usepackage{tikz}
\usepackage{epsfig}
\usepackage{subfigure}

\providecommand{\U}[1]{\protect\rule{.1in}{.1in}}
\allowdisplaybreaks 

\newtheorem{theorem}{Theorem}[section]

\newtheorem{corollary}[theorem]{Corollary}
\newtheorem{definition}[theorem]{Definition}

\def\]{{\Big]}}
\def\[{{\Big[}}

\def\bd{\begin{definition}}
\def\ed{\end{definition}}
\def\bp{\begin{proposition}}
\def\ep{\end{proposition}}
\def\bc{\begin{corollary}}
\def\ec{\end{corollary}}
\def\bx{\begin{Examples}}
\def\ex{\end{Examples}}

\def\ba{{\begin{align}}
\def\ea{\end{align}}}

\def\geq{\geqslant}
\def\leq{\leqslant}

\def\^{\widehat}

\def\ba{\begin{aligned}}
\def\ea{\end{aligned}}
\def\be{\begin{equation}}
\def\ee{\end{equation}}
\def\ben{\begin{align*}}
\def\enn{\end{align*}}

\makeatletter

\newcommand{\Rmnum}[1]{\expandafter\@slowromancap\romannumeral #1@}
\makeatother

\numberwithin{equation}{section}

\title{\bf{The average genus for bouquets of circles and  dipoles}}

{\author{ Jinlian Zhang$^\dag$, Xuhui Peng$^\spadesuit$, Yichao Chen$^\dag$ \\
{\em\small $^\dag$College of Mathematics and Econometrics, Hunan University, 410082 Changsha, China}\\
{\em\small $^\spadesuit$MOE-LCSM, School of mathematics and statistics,  Hunan Normal University }\\
{\em\small Changsha  410081, China}\\
 }
 }

 \date{}

\begin{document}

\maketitle

\let\thefootnote\relax\footnotetext{$^\spadesuit$
Corresponding author.}

 \let\thefootnote\relax\footnotetext{
  Email:jinlian916@hnu.edu.cn(J.Zhang),xhpeng@hunnu.edu.cn(X.Peng),yichen@hnu.edu.cn(Y.Chen)
 }

\begin{abstract}
\noindent The bouquet of circles $B_n$ and dipole graph $D_n$ are two important classes of graphs in topological graph theory.
For  $n\geq 1$, we  give an explicit   formula for  the  average genus $\gamma_{avg}(B_n)$ of $B_n$.
By this   expression, one easily sees $\gamma_{avg}(B_n)=\frac{n-\ln n-c+1-\ln 2}{2}+o(1)$, where $c$ is the Euler constant.   Similar results are obtained for   $D_n$.
Our method is new  and deeply  depends on    the  knowledge  in  ordinary differential equations.
\vspace{1mm}
\end{abstract}
\section{Introduction and Main Results}
A \emph{graph} $G = (V(G), E(G))$ is permitted to have
both loops and multiple edges. A \textit{embedding} of a graph  $G$  into an orientable surface $O_k$ is a \textit{{cellular} embedding}, i.e., the interior of every face is homeomorphic to an open disc.
We denote the number of {cellular}  embeddings of $G$ on the surface $O_k$ by $g_k(G)$, where, by the \textit{number of embeddings}, we mean the number of equivalence classes under ambient isotopy.
The \emph{genus polynomial} of a graph  $G$  is given by
$$\Gamma_{G}(x)=\sum_{k\geq 0}g_k(G)x^k,$$
 This sequence  $\{g_k(G),k=0,1,2,\cdots,\}$ is called the \emph{genus distribution} of the graph $G$.

 The \emph{average genus} $\gamma_{avg}(G)$  of the graph $G$ is the expected value of the genus random variable, over all labeled 2-cell orientable embeddings of $G$, using the uniform distribution.
 In other words, the average genus of  $G$ is
  \begin{eqnarray*}
    \gamma_{avg}(G)&=&\frac{\Gamma_{G}'(1)}{\Gamma_{G}(1)}= \sum_{k=0}^{\infty }k\cdot \frac{g_k(G)}{\Gamma_{G}(1)}.
    \end{eqnarray*}
 The  study of the average genus of a graph  began by
Gross and Furst \cite{GF87},
and much further developed by Chen and Gross \cite{CJ92,CJ93,Chen94}.
 Two lower bounds were  obtained in \cite{CGR}  for the average genus of  two kinds of  graphs.
  In \cite{SS91}, Stahl gave the asymptotic result for average genus of linear graph families.
   The
  exact value for the average genus of  small-order
complete graphs,  closed-end ladders, and cobblestone paths were   derived   by White \cite{White}.
More  references are the following:
\cite{CLH,GKR93,MT16,SS83,SS95}  etc.  For general background in topological graph theory, we refer the readers to see Gross and Tucker \cite{GT87} or White \cite{WAT84}.

One   objective of this paper is to  give an explicit   expression of  the average genus for    a bouquet of circles.
By a  \emph{bouquet of circles}, or more briefly, a bouquet, we mean a graph with one vertex and some self-loops.    In particular, the bouquet with  $n$ self-loops is denoted $B_n$.
 Figure \ref{fig1} demonstrates the graphs $B_1,B_2,B_3.$
 The bouquets $\{B_n,n\geq 1\}$ are  very important graphs in 	topological graph theory.
First,  since any connected graph can be reduced to a bouquet by contracting a spanning tree to a point,  bouquets are fundamental building blocks of topological graph theory.
Second, as that  demonstrated  in  \cite{GJ77,GT77}, Cayley graphs and many other regular graphs are covering spaces of bouquets.

\begin{figure}[h]
  \centering
  \includegraphics[width=0.50\textwidth]{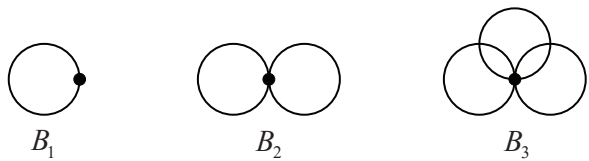}
  \caption{The bouquets $B_{1},B_{2},$ and $B_{3}.$}
  \label{fig1}
\end{figure}
For the genus distribution of $B_n$,
Gross, Robbins and  Tucker \cite{GR89} proved that  the numbers $g_m(B_n)$ of embeddings of the  $B_n$ in the oriented surface of genus $m$  satisfy the following   recurrence for $n>2,$
\begin{equation}
\label{1-3}
\begin{split}
(n+1)g_{m}(B_n)&=4(2n-1)(2n-3)(n-1)^2(n-2)g_{m-1}(B_{n-2})
 \\ &  \quad +4(2n-1)(n-1)g_{m}(B_{n-1})
 \end{split}
 \end{equation}
 with initial conditions
 \begin{eqnarray*}
&&  g_m(B_0)=1 \text{ for } m=0   \text{ and } g_m(B_0)=1 \text{ for } m>0,
 \\ &&
  g_m(B_1)=1 \text{ for } m=0  \text{ and } g_m(B_1)=1 \text{ for } m>0.
 \end{eqnarray*}
 With an aid of edge-attaching surgery technique, the  total embedding polynomials of $B_n$ was computed in  \cite{KS02}.
 Stahl \cite{SS90} also  did some researches on the average genus of $B_n$.
 By  \cite[Theorem 2.5]{SS90} and Euler formula,  one easily sees that
 \begin{eqnarray}
 \label{3-1}
   \lim_{n\rightarrow \infty}\Big(\gamma_{avg}(B_n)-\big(\frac{n+1}{2}-\frac{1}{2}\sum_{k=1}^{2n}\frac{1}{k}
   \big)\Big)=0.
 \end{eqnarray}
To achieve this,  Stahl
made many accurate estimates on the unsigned Stirling numbers $s(n,k)$ of the first kind.
In this paper, using    knowledge  in  ordinary differential equations and  Taylor formulas, we derive an  explicit   expression of $\gamma_{avg}(B_n).$
 By this expression,  (\ref{3-1}) follows 	immediately.
 Our methods are totally different from   that in \cite{SS90} and we don't need to make estimates on $s(n,k)$.
 In Section 2, we will give the computation of $\gamma_{avg}(B_n)$ in detail.

Another    objective of this paper is to  give an explicit   expression of  the average genus for    \emph{dipoles} $D_n$  (two vertices, $n-$multiple edges).
Figure \ref{fig2} demonstrates the graphs $D_1,D_2,D_3.$

\

\begin{figure}[h]
  \centering
  \includegraphics[width=0.35\textwidth]{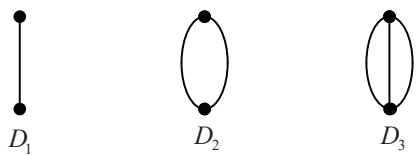}
  \caption{The dipoles $D_{1},D_{2},$ and $D_{3}.$}
  \label{fig2}
\end{figure}

The dipole, like the bouquet, is useful as a voltage graph. See \cite{WAT84} for example.
Moreover, hypermaps  correspond with the 2-cell embeddings of the dipole.
The genus distributions of $D_n$ are given by
  \cite{KS02} and \cite{Ri87}.
  In this paper, the calculation of $\gamma_{avg}(D_n)$ are similar as that in $\gamma_{avg}(B_n).$
  But the processes  are  more complicated, so we  still give their  details in Section 3.

It is not difficult to obtain    the following recurrence relation for $\gamma_{avg}(B_n)$
  \begin{eqnarray*}
  \gamma_{avg}(B_n)&=& \frac{2}{n+1}\gamma_{avg}(B_{n-1})+\frac{n-1}{n+1}\big(\gamma_{avg}(B_{n-2})+1\big).
\end{eqnarray*}
Since the coefficients  before $\gamma_{avg}(B_{n-2})$
and
$\gamma_{avg}(B_{n-1})$ depend  on $n$,
there are little tools  to obtain the explicit   expression of $\gamma_{avg}(B_n)$.
In this paper, we give a method to get explicit   expression of $\gamma_{avg}(B_n)$.
Our method is new  and deeply  depends on    the  knowledge  in  ordinary differential equations and series theory.
\section{The average genus of $B_n$}
The main objective of this section  is to prove the following  theorem.
\begin{theorem}
\label{1-2}
The average genus of $B_n$ is given by
 \begin{eqnarray*}
  &&\gamma_{avg}(B_n)=\frac{n+1}{2}-\sum_{m=0}^{n-1} \frac{1+(-1)^m}{2(m+1)}-\frac{1+(-1)^n }{4(n+1)}.
\end{eqnarray*}
In particular,  we have   $$\gamma_{avg}(B_n)=\frac{n-\ln n-c+1-\ln 2}{2}+
o(1),$$ where $c\approx 0.5772$ is the Euler constant.
\end{theorem}
\begin{proof}
Multiplying both sides of (\ref{1-3})
 by $x^m$,  it holds that
 \begin{equation}
\begin{split}
\sum_{m\geq 0}(n+1)g_{m}(B_n)x^m&=\sum_{m\geq 0}4(2n-1)(2n-3)(n-1)^2(n-2)g_{m-1}(B_{n-2})x^m
 \\ &  \quad +\sum_{m\geq 0}4(2n-1)(n-1)g_{m}(B_{n-1})x^m.
 \end{split}
 \end{equation}
 Hence,  the  genus polynomial $\Gamma_{B_n}(x)$ satisfy the following   recurrence
\begin{eqnarray}
\label{1-4}
\begin{split}
(n+1)\Gamma_{B_n}(x)& =4(2n-1)(2n-3)(n-1)^2(n-2)\cdot x \cdot \Gamma_{B_{n-2}}(x)
\\ & \quad  +4(2n-1)(n-1)\Gamma_{B_{n-1}}(x)
\end{split}
\end{eqnarray}
with initial conditions $ \Gamma_{B_1}(x)=1,\Gamma_{B_2}(x)=4+2x.$

By (\ref{1-4}),  we have
\begin{eqnarray*}
\label{1-6}
&&  (n+1)\Gamma_{B_n}'(1) =4(2n-1)(2n-3)(n-1)^2(n-2) \cdot \Gamma_{B_{n-2}}'(1)
\\ && ~~~+4(2n-1)(2n-3)(n-1)^2(n-2)\cdot \Gamma_{B_{n-2}}(1)+4(2n-1)(n-1)\Gamma_{B_{n-1}}'(1).
\end{eqnarray*}
Dividing both sides of the above equality by $\Gamma_{B_n}(1)=(2n-1)!$, for $n\geq 3,$ one arrives at that
\begin{eqnarray}
\label{1-5}
  (n+1)\gamma_{avg}(B_n)\!&=& \! 2\gamma_{avg}(B_{n-1})+(n-1)\big(\gamma_{avg}(B_{n-2})+1\big).
\end{eqnarray}
By direct calculation, one sees
$\gamma_{avg}(B_1)=0, \gamma_{avg}(B_2)=\frac{1}{3}.$
For $n\leq 0$, we define $\gamma_{avg}(B_n)=0$ so that  (\ref{1-5}) holds for any integer  $n\geq 1.$
For simplicity  of  writing,  we use  $a_n$
to denote $\gamma_{avg}(B_n)$ in the rest of proof.
Thus, $a_n$ satisfies the following recurrence relation
\begin{eqnarray}
\label{1-1}
  a_n=\frac{2}{n+1}a_{n-1}+\frac{n-1}{n+1}(a_{n-2}+1)
\end{eqnarray}
with initial conditions $a_1=0, a_2=\frac{1}{3}.$

Using   $(\ref{1-1})$, we have
\begin{eqnarray*}
 (n+1)a_n=2a_{n-1}+(n-1)(a_{n-2}+1)
\end{eqnarray*}
and
\begin{eqnarray}
\label{1-10}
 \sum_{n\geq 1}(n+1)a_n t^n=2\sum_{n\geq 1}a_{n-1}t^n+\sum_{n\geq 1}(n-1)(a_{n-2}+1)t^n.
\end{eqnarray}
Let
$
u(t)=\sum_{n\geq 1}a_n t^n.
$
Then, by (\ref{1-10}),  we obtain
\begin{eqnarray*}
 \Big(t\cdot \sum_{n\geq 1}a_n t^n\Big)'&=&2t\cdot \sum_{n\geq 1}a_{n-1}t^{n-1}+\sum_{n\geq 1}(n-2)a_{n-2}t^n+\sum_{n\geq 1}a_{n-2}t^n+\sum_{n\geq 1}(n-1)t^n
 \\ &=& 2tu(t)+t^3 \sum_{n\geq 1}(n-2)a_{n-2}t^{n-3}+t^2u(t)+t^2 \cdot ( \sum_{n\geq 2}t^{n-1})',
\end{eqnarray*}
that is
\begin{eqnarray*}
  (tu(t))'&=& 2tu(t)+t^3\sum_{n\geq 3}(n-2)a_{n-2}t^{n-3}+t^2u(t)+t^2\Big(\frac{t}{1-t}\Big)'
  \\ &=&
  2tu(t)+t^3\sum_{n\geq 1}na_{n}t^{n-1}+t^2u(t)+t^2\Big(\frac{t}{1-t}\Big)'
  \\ &=&
  2tu(t)+t^3u'(t)+t^2u(t)+t^2\Big(\frac{t}{1-t}\Big)',
\end{eqnarray*}
which implies that   $u(t)$ satisfies the following
first order  linear  differential equation
\begin{eqnarray}
\label{e-1}
  u'(t)(t-t^3)+u(t)(1-2t-t^2)=\frac{t^2}{(1-t)^2}.
\end{eqnarray}
with initial condition $u(t)|_{t=0}=0.$
We solve this equation  using variation of parameters and obtain
its    solution
\begin{eqnarray*}
 u(t)=\frac{-\left(t^2-1\right) \ln (1-t)+\left(t^2-1\right) \ln (t+1)+2 t}{4 (t-1)^2 t}.
\end{eqnarray*}
Denote
\begin{eqnarray*}
  u_1(t)=\frac{1}{2 (t-1)^2},~u_2(t)=-\frac{(t+1) \ln (1-t)}{4 (t-1) t},~u_3(t)=\frac{(t+1) \ln (t+1)}{4 (t-1) t}.
\end{eqnarray*}
Then, we have
$
  u(t)=u_1(t)+u_2(t)+u_3(t).
$
By Taylor formula,  we obtain
\begin{eqnarray}
\label{c-1}
  u_1(t)&=&\sum_{n\geq 0}\frac{n+1}{2}t^n
\end{eqnarray}
and
\begin{eqnarray}
 \nonumber   u_2(t) &=& \frac{1}{4}(1+t)\cdot\frac{1}{1-t}\cdot
  \frac{\ln (1-t)}{t}
 =  \frac{1}{4}(1+t)\cdot \sum_{\ell\geq 0}t^{\ell}\cdot \sum_{m\geq 0}\big(-\frac{1}{m+1}t^m\big)
  \\  \label{1-8} &=& \frac{1}{4}(1+t)\cdot \sum_{n\geq 0} \sum_{m= 0}^n (-\frac{1}{m+1})t^n
  = \sum_{n\geq 0}b_nt^n,
\end{eqnarray}
where $b_0=-\frac{1}{4}$ and
$
b_n=\frac{1}{4}\Big[\sum_{m= 0}^n (-\frac{1}{m+1})+\sum_{m= 0}^{n-1} (-\frac{1}{m+1})\Big], n\geq 1.
$
For $u_3(t),$ we have
\begin{eqnarray}
 \nonumber u_3(t) &=& -\frac{1}{4}(1+t)\cdot\frac{1}{1-t}\cdot
  \frac{\ln (1+t)}{t}
 =  -\frac{1}{4}(1+t)\cdot \sum_{\ell\geq 0}t^{\ell}\cdot \sum_{m\geq 0}\frac{(-1)^m}{m+1}t^m
    \\  \label{1-9} &=& -\frac{1}{4}(1+t)\cdot \sum_{n\geq 0} \sum_{m= 0}^n \frac{(-1)^m }{m+1}t^n
 =\sum_{n\geq 0}c_nt^n,
\end{eqnarray}
where   $c_0=-\frac{1}{4}$ and
\begin{eqnarray*}
  c_n&=& -\frac{1}{4}\Big[\sum_{m= 0}^n \frac{(-1)^m}{m+1}+\sum_{m= 0}^{n-1} \frac{(-1)^m }{m+1})\Big],\quad n\geq 1.
\end{eqnarray*}
Combining  (\ref{c-1})-(\ref{1-9}), it holds that
\begin{eqnarray*}
a_n&=& \frac{n+1}{2}+b_n+c_n
  \\ &=& \frac{n+1}{2}+\frac{1}{4}\Big[\sum_{m= 0}^n (-\frac{1}{m+1})+\sum_{m= 0}^{n-1} (-\frac{1}{m+1})\Big]-\frac{1}{4}\Big[\sum_{m= 0}^n \frac{(-1)^m}{m+1}+\sum_{m= 0}^{n-1} \frac{(-1)^m }{m+1})\Big],
\end{eqnarray*}
which completes the proof of Theorem \ref{1-2}.
\end{proof}
\section{The average genus of $D_n$}
The main objective of this section  is to prove the following  theorem.
\begin{theorem}
\label{f-1}
$\gamma_{avg}(D_1)=\gamma_{avg}(D_2)=0$ and  for $n\geq 3$,  we have
\begin{eqnarray}
\nonumber && \!\gamma_{avg}(D_n)\!=\!
\sum_{m=4}^{n+1}\frac{ \Big(4 (-1)^m m^2+m^2-12 (-1)^m m-5 m+6 (-1)^m+6\Big)\cdot (n-m+2)}{2(m-3) (m-2) (m-1) m}
\\  \label{b-1}
\end{eqnarray}
In particular,
we have
  \begin{eqnarray}
  \label{b-2}
    \gamma_{avg}(D_n)=\frac{n-\ln n-c }{2}+o(1),
  \end{eqnarray}
  where $c\approx 0.5772$ is the Euler constant.
\end{theorem}
\begin{proof}
Once we have obtained (\ref{b-1}), by the soft \emph{Mathematica}  or the series theory, we have
 \begin{eqnarray}
 \label{d-2}
 && \sum_{m=4}^{n+1}\frac{ 4 (-1)^m m^2+m^2-12 (-1)^m m-5 m+6 (-1)^m+6}{(m-3) (m-2) (m-1) m}=1-\frac{1}{n}+o(\frac{1}{n})
 \end{eqnarray}
 and
 \begin{eqnarray}
   \nonumber  && \sum_{m=4}^{n+1}\frac{ \Big(4 (-1)^m m^2+m^2-12 (-1)^m m-5 m+6 (-1)^m+6\Big)\cdot  (m-2)}{(m-3) (m-2) (m-1) m}
 \\ \nonumber && =\sum_{m=4}^{n+1}\frac{ \Big(4 (-1)^m m^2-12 (-1)^m m-5 m+6 (-1)^m+6\Big)\cdot  (m-2)}{(m-3) (m-2) (m-1) m}
   \\ \nonumber
   &&\quad +\sum_{m=4}^{n+1}\frac{m^2\cdot (m-2)}{(m-3) (m-2) (m-1) m}
   \\  \nonumber &&=-\frac{7}{4}+o(1)+\sum_{m=4}^{n+1}
   \frac{m}{(m-3)  (m-1)}
   \\  \nonumber &&=-\frac{7}{4}+o(1)+\sum_{m=4}^{n+1}
   \frac{1}{(m-3) }
    +\sum_{m=4}^{n+1}
   \frac{1}{(m-3) (m-1)}
     \\      \label{d-1}  &&=-\frac{7}{4}+o(1)+\sum_{m=4}^{n+1}
   \frac{1}{(m-3) }
    +\frac{3}{4}
         =-1+o(1)+\ln n+c,
  \end{eqnarray}
  where $c$ is the Euler constant.
  Combining (\ref{d-2})(\ref{d-1}), we complete the proof of (\ref{b-2}).

 Now we give a proof of (\ref{b-1}). By \cite[Theorem 5.2]{Ri87}, we have
\begin{eqnarray*}
  (n+2)g_k(D_{n+1})=n (2n+1) g_k(D_n)+n^3(n-1)^2g_{k-1}(D_{n-1})-n(n-1)^2g_k(D_{n-1})
\end{eqnarray*}
and
\begin{eqnarray*}
   \sum_{k\geq 0}(n+2)g_k(D_{n+1})x^k&=&n (2n+1) \sum_{k\geq 0}g_k(D_n)x^k+n^3(n-1)^2\sum_{k\geq 0}g_{k-1}(D_{n-1})x^k
  \\ && -n(n-1)^2\sum_{k\geq 0}g_k(D_{n-1})x^k,
\end{eqnarray*}
where $g_k(D_n)=0$ for $k=-1.$
Therefore, $\Gamma_{D_n}(x)$ satisfies   the following recurrence relation
  \begin{eqnarray}
  \label{a-1}
(n+2)\Gamma_{D_{n+1}}(x)-  n(2n+1)\Gamma_{D_{n}}(x)
  =n(n-1)^2 (n^2x-1) \cdot \Gamma_{D_{n-1}}(x)
\end{eqnarray}
with initial conditions
$\Gamma_{D_1}(x)=\Gamma_{D_2}(x)=1.$
Taking derivations on both sides of (\ref{a-1}) and setting $x=1$, one sees that
\begin{eqnarray*}
(n+2)\Gamma_{D_{n+1}}'(1)=  n(2n+1)\Gamma_{D_{n}}'(1)
  +n(n-1)^2 (n^2-1) \cdot \Gamma_{D_{n-1}}'(1)
  +n^3(n-1)^2  (n-2)!^2.
\end{eqnarray*}
Dividing both sides of the above equality  by $n!^2$, we obtain
  \begin{eqnarray*}
(n+2)\gamma_{avg}(D_{n+1})=  \frac{2n+1}{n}\gamma_{avg}(D_{n})
  + \frac{(n^2-1)}{n} \cdot \gamma_{avg}(D_{n-1})
  +n
\end{eqnarray*}
and
  \begin{eqnarray*}
n(n+2)\gamma_{avg}(D_{n+1})=  (2n+1)\gamma_{avg}(D_{n})
  + (n^2-1)\cdot \gamma_{avg}(D_{n-1})
  +n^2.
\end{eqnarray*}
 For simplicity  of  writing,  we use  $a_n$
to denote $\gamma_{avg}(D_n)$ in the rest of proof. Then, for  $n\geq 2$,  we have the following  recurrence relation for $a_n$
  \begin{eqnarray}
  \label{b-4}
&& n(n+2)a_{n+1}=  (2n+1)a_{n}
  + (n^2-1)a_{n-1}
  +n^2
\end{eqnarray}
with initial conditions $a_1=a_2=0.$

Let
$$
u(t)=\sum_{n\geq 1}a_nt^{n-3}=\sum_{n\geq 2}a_{n+1}t^{n-2}.
$$
Then, by (\ref{b-4}) and calculation, we have
 \begin{eqnarray}
 \label{b-5}
&& \sum_{n\geq 2}\!n(n+2)a_{n+1}t^{n-2}\!= \! \sum_{n\geq 2}\! (2n+1)a_{n}t^{n-2}
  + \!\sum_{n\geq 2}\!(n^2-\!1)a_{n-1}t^{n-2}
  +\! \sum_{n\geq 2}\! n^2t^{n-2}
\end{eqnarray}
and
\begin{eqnarray*}
  u'(t)&=&\sum_{n\geq 4}a_n (n-3)t^{n-4}=\sum_{n\geq 3}a_{n+1} (n-2)t^{n-3}=\sum_{n\geq 2}a_{n+1} (n-2)t^{n-3},
  \\ u''(t) &=& \sum_{n\geq 5}a_{n} (n-3)(n-4)t^{n-5}= \sum_{n\geq 4}a_{n+1} (n-2)(n-3)t^{n-4}=\sum_{n\geq 2}a_{n+1} (n-2)(n-3)t^{n-4}.
\end{eqnarray*}

Also, by calculation  and the above two equalities, we obtain
\begin{eqnarray*}
  \sum_{n\geq 2}n(n+2)a_{n+1}t^{n-2}
&=&\sum_{n\geq 2}\big[n^2-5n+6+7(n-2)+8\big]a_{n+1}t^{n-2}
  \\ &=& t^2 u''(t)+7t u'(t)+8u(t),
  \\
  \sum_{n\geq 2}(2n+1)a_{n}t^{n-2}&=&\sum_{n\geq 2}(2n+3)a_{n+1}t^{n-1}
  = \sum_{n\geq 2}\big(2(n-2)+7\big)a_{n+1}t^{n-1}
   \\ &=&  2t^2 u'(t)+7t u(t),
   \\
   \sum_{n\geq 2}(n^2-1)a_{n-1}t^{n-2}&=&\sum_{n\geq 4}(n^2-1)a_{n-1}t^{n-2}
=\sum_{n\geq 2}(n^2+4n+3)a_{n+1}t^{n}
  \\ &=&\sum_{n\geq 2}\Big[n^2-5n+6+9(n-2)+15\Big]a_{n+1}t^{n}
  \\& =& t^4u''(t)+9t^3 u'(t)+15t^2u(t),
 \\ \sum_{n\geq 2}n^2t^{n-2}&=& \sum_{n\geq 2}n(n-1)t^{n-2}+\sum_{n\geq 2}nt^{n-2}
 =v''(t)+\sum_{n\geq 0}nt^{n-2}-t^{-1}
  \\ &=&v''(t)+\frac{v'(t)}{t}-t^{-1} =\frac{3t-4-t^2}{(t-1)^3},
\end{eqnarray*}
where  $v(t)=\sum_{n\geq 0}t^n, v'(t)=\sum_{n\geq 0}nt^{n-1},v''(t)=\sum_{n\geq 0}n(n-1)t^{n-2}.$
Substituting the above equalities into (\ref{b-5}), by the definition of $u(t),$   one arrives at that $u(t)$ satisfies the following second  order linear  differential equation
\begin{eqnarray*}
 (t^2-t^4) u''(t)+(7t-2t^2-9t^3)u'(t)+(8-7t-15t^2)u(t)=\frac{3t-4-t^2}{(t-1)^3}
\end{eqnarray*}
with initial conditions
$u(0)=a_3=\gamma_{avg}(D_3)=\frac{1}{2}, u'(0)=a_4=\gamma_{avg}(D_4)=\frac{5}{6}.
$

With the help of a computer, the solution to the above equation is
\begin{eqnarray*}
&& u(t)=\frac{1}{4 (t-1) t^2}+\frac{v(t)}{4 (t-1)^2 t^4},
\end{eqnarray*}
where
\begin{eqnarray*}
v(t)&=& -t^3+2 t^3 \ln (t+1)+3 t^2-2 t^2 \ln (t+1)
\\ && -2 t \ln (1-t)-2 t \ln (t+1)+2 \ln (1-t)+2 \ln (t+1).
\end{eqnarray*}
By taylor formula, we have
\begin{eqnarray*}
  && \frac{1}{4 (t-1) t^2}= \sum_{m\geq -2}(-\frac{1}{4})t^m,
\\ &&  v(t)=t^2-t^3+\sum_{m\geq 4}\frac{2 \Big(4 (-1)^m m^2+m^2-12 (-1)^m m-5 m+6 (-1)^m+6\Big)}{(m-3) (m-2) (m-1) m}t^m,
\\
  && \frac{1}{4 (t-1)^2 t^4}=\sum_{m\geq -4}\frac{m+5}{4}t^m.
\end{eqnarray*}
Therefore,
\begin{eqnarray*}
  a_{n+3}&=&-\frac{1}{4}+\frac{n+3}{4}-\frac{n+2}{4}
  \\ &&+\sum_{m=4}^{n+4}\frac{2 \Big(4 (-1)^m m^2+m^2-12 (-1)^m m-5 m+6 (-1)^m+6\Big)}{(m-3) (m-2) (m-1) m}\cdot \frac{n-m+5}{4}
  \\ &=&\sum_{m=4}^{n+4}\frac{ 4 (-1)^m m^2+m^2-12 (-1)^m m-5 m+6 (-1)^m+6}{(m-3) (m-2) (m-1) m}\cdot \frac{n-m+5}{2}
  \\ &=& \frac{n}{2}\sum_{m=4}^{n+4}\frac{ 4 (-1)^m m^2+m^2-12 (-1)^m m-5 m+6 (-1)^m+6}{(m-3) (m-2) (m-1) m}
 \\ &&  -\sum_{m=4}^{n+4}\frac{ 4 (-1)^m m^2+m^2-12 (-1)^m m-5 m+6 (-1)^m+6}{(m-3) (m-2) (m-1) m}\cdot  \frac{m-5}{2},
\end{eqnarray*}
which yields the desired result (\ref{b-1}).
\end{proof}

\section{Some remarks}

Bouquets and dipoles are two important classes of graphs in topological graph theory. Their average genera are of independent interests. In this paper, we obtain explicit formulas for $\gamma_{avg}(B_n)$ and  $\gamma_{avg}(D_n).$ By Theorems \ref{1-2}, \ref{f-1}, we have the following relation between $\gamma_{avg}(B_n)$ and  $\gamma_{avg}(D_n),$
 \begin{eqnarray*}
    \gamma_{avg}(B_n)=\gamma_{avg}(D_n)+\frac{1-\ln 2}{2}+o(1).
  \end{eqnarray*}
  It follows that the difference of  $\gamma_{avg}(B_n)$ and   $\gamma_{avg}(D_n)$
   tends  to a constant  $\frac{1-\ln 2}{2}$ when $n$ tends to infinity.

Since both $B_n$ and $D_n$ are upper-embeddable,  the  maximum genera of $B_n$  and $D_n$  are  $\left\lfloor \frac{n}{2}\right\rfloor$ and $\left\lfloor\frac{n-1}{2}\right\rfloor$, respectively. Recall that the
 minimum genera of $B_n$ and $D_n$ equal $0.$
 Therefore, also by Theorems \ref{1-2}, \ref{f-1},   we have
 \begin{eqnarray*}
   \lim_{n\rightarrow \infty}\frac{\gamma_{avg}(B_n)}{\lfloor \frac{n}{2}\rfloor }=1 \text{ and } \lim_{n\rightarrow \infty}\frac{\gamma_{avg}(D_n)}{\lfloor \frac{n-1}{2}\rfloor }=1.
 \end{eqnarray*}
This   imply  that   the average genus of  $B_n$ $(D_n)$ is more closer to the maximum genus than to the minimum genus.

\end{document}